\documentclass{amsart}

\newtheorem{theorem}{Theorem}[section]
\newtheorem{lemma}[theorem]{Lemma}

\theoremstyle{definition}

\theoremstyle{remark}

\numberwithin{equation}{section}

\begin{document}

\title{The Sym(3) Conjecture and Alt(8)}

\author{Cecil Andrew Ellard}
\address{Bloomington, Indiana}
\curraddr{}
\email{cellard@ivytech.edu}






\begin{abstract}
We give an alternate computer-free proof of a result of Z. Arad, M. Muzychuk, and A. Oliver: if $G$ is a minimal counterexample to the $Sym(3)$ conjecture, then $Soc(G)^{\prime }$ cannot be isomorphic to $Alt(8)$.
\end{abstract}

\maketitle








\section{The $Sym(3)$ Conjecture}

The $Sym(3)$ conjecture states that if $G$ is a non-trivial finite group whose conjugacy classes have distinct cardinalities, then $G$ is isomorphic to $Sym(3)$, the symmetric group of order $6$. With additional assumptions on $G$, the conclusion is known to be true: for example when $G$ is supersolvable [Markel, 1973] or when $G$ is solvable [Zhang, 1994]. Let $\mathcal{H}$ be the hypothesis of the $Sym(3)$ conjecture:\\

\noindent ($\mathcal{H}$) The group $G$ is a finite, non-trivial group whose conjugacy classes have distinct cardinalities.\\
\\
The authors Z. Arad, M. Muzychuk, and A. Oliver [2004] proved that if $G$ satisfies $\mathcal{H}$, then either $Soc(G)^{\prime }=1$ or  $Soc(G)^{\prime }$ is isomorphic to one of the following:\\

\noindent 1. $ Alt(5)^{a}$, for $1 \leq a \leq 5, a \neq 2$\\
2. $Alt(8)$\\
3. $PSL(3,4)^{e}$, for $1 \leq e \leq 10$\\
4. $Alt(5) \times PSL(3,4)^{e}$, for $1 \leq e \leq 10$\\

\noindent The authors go on to show that if $G$ is a minimal counterexample to the $Sym(3)$ conjecture, then the list of possibilities for $Soc(G)^{\prime }$ can be shortened; either $Soc(G)^{\prime }=1$ or  $Soc(G)^{\prime }$ is isomorphic to one of the following: \\

\noindent 1. $Alt(5)^{a}$, for $3 \leq a \leq 5$\\
2. $PSL(3,4)^{e}$, for $1 \leq e \leq 10$\\

\noindent So in particular, in a minimal counterexample to the $Sym(3)$ conjecture, $Soc(G)^{\prime }$ cannot be isomorphic to $Alt(8)$. Their proof makes use of computer programs in the algebra system GAP. The purpose of this note is to give an alternate computer-free proof of their result for $Alt(8)$:\\

 \noindent \textbf{Theorem:} Let $G$ be a minimal counterexample to the $Sym(3)$ conjecture. Then $Soc(G)^{\prime}$ is not isomorphic to $Alt(8)$.\\

In what follows, we will let $Sym(n)$ be the symmetric group on $n$ letters, and $Alt(n)$ the alternating group on $n$ letters. For any group $G$, we will let $Soc(G)$ be the socle of $G$ (the subgroup generated by the minimal normal subgroups of $G$), $Z(G)$ the center of $G$,  $G^{\prime }$ the derived subgroup of $G$, $Aut(G)$ the automorphism group of $G$, $Inn(G)$  inner automorphism group of $G$, and $Out(G)$ the outer automorphism group of $G$. If $X$ is a subset of $G$, then $C_{G}(X)$ will be the centralizer of $X$ in $G$, and  $N_{G}(X)$ will be the normalizer of $X$ in $G$. If $g \in G$, we will let $g^{G}$ denote the set of all elements of $G$ that are conjugate in $G$ to $g$. In other words, $g^{G}$ is the conjugacy class in $G$ containing $g$.  Euler's totient function will be denoted by $\phi $.\\

\noindent A finite group is said to be a \textit{rational group} if every ordinary (i.e. complex) character is rational-valued. That is, for every ordinary character $\chi $ of $G$ and for every $g \in G$, we have $\chi (g) \in \mathbb{Q}$.  This is equivalent to only requiring that the irreducible ordinary characters of $G$ are rational-valued. We have the following characterization of rational groups (see [Isaacs], [Serre], or [Ellard, 2013]):

\begin{lemma}
Let $G$ be a finite group. Then the following are equivalent:\\

\noindent (i) $G$ is a rational group\\
(ii) For all $g \in G$, all generators of $\langle g \rangle$ are conjugate in $G$.\\
(iii) For all $g \in G$, $[N_{G}(\langle g \rangle ){:}C_{G}(\langle g \rangle )]=\phi (ord(g))$.\\
\end{lemma}

\noindent The finite symmetric groups $Sym(n)$ for $n \geq 1$ are rational groups [Isaacs]. Also, any group satisfying hypothesis $\mathcal{H}$ is a rational group:

\begin{lemma}
Let $G$ be a group which satisfies hypothesis $\mathcal{H}$. Then $G$ is a rational group.
\end{lemma}

\begin{proof}
Let $G$ be a group which satisfies hypothesis $\mathcal{H}$, and let $g \in G$. By the above lemma, it suffices to show that all generators of $\langle g \rangle $ are conjugate in $G$. Assume that $g_{1}$ and $g_{2}$ are elements of $G$ which generate $\langle g \rangle $. We wish to show that  $g_{1}$ and $g_{2}$ are conjugate. Note that $C_{G}(\langle g_{1} \rangle )=C_{G}(\langle g_{2} \rangle )$. Therefore\\

 $|g_{1}^{G}|=[G{:}C_{G}(g_{1})]=[G{:}C_{G}(\langle g_{1} \rangle  )]=[G{:}C_{G}(\langle g_{2} \rangle  )]=[G{:}C_{G}(g_{2})]=|g_{2}^{G}|$\\

\noindent and therefore the conjugacy classes $g_{1}^{G}$ and $g_{2}^{G}$ have the same cardinalities. So by hypothesis $\mathcal{H}$, $g_{1}^{G}$ and $g_{2}^{G}$ are the same conjugacy class, so  $g_{1}$ and $g_{2}$ are conjugate in $G$.
\end{proof}

We also will make use of the fact that if $H$ and $K$ are finite groups, then $H \times K$ is rational if and only if both $H$ and $K$ are rational. This is true because the irreducible ordinary characters of $H \times K$ are precisely the functions of the form $(h,k) \mapsto \chi_{1}(h) \cdot \chi_{2}(k)$ where $\chi_{1}$ and $\chi_{2}$ are irreducible ordinary characters of $H$ and $K$, respectively.   [Isaacs]\\

\noindent The term Frattini's Argument refers to the theorem stating that if G is a finite group and if $N$ is a normal subgroup of $G$, and if $P$ is a Sylow $p$-subgroup of $N$, then $G=N_{G}(P)N$. This theorem is often generalized as follows: if $G$ is a finite group acting transitively on a set $\Omega $, and if $N$ is a normal subgroup of $G$ which is also transitive on $\Omega $, then for any $\omega \in \Omega $,  $G=Stab_{G}(\omega )N$. One can go a little further with the following: \newpage

\begin{lemma} (Extended Frattini Argument) Let $G$ be a finite group acting (not necessarily transitively) on a set $\Omega $,  let $N$ be a normal subgroup of $G$ and let $\omega \in \Omega $. Then:\\

(i) $|\omega^{G}|/|\omega^{N}|=[G{:}Stab_{G}(\omega )N]$

(ii)  $|\omega^{G}|/|\omega^{N}|=1$ iff $G=Stab_{G}(\omega )N$

(iii)  $|\omega^{G}|/|\omega^{N}|=[G{:}N]$ iff $Stab_{G}(\omega) \leq N$

(iv) $|\omega^{G}|/|\omega^{N}|< [G{:}N]$ iff $Stab_{G}(\omega)$ is not a subgroup of $N$
\end{lemma}

\begin{proof}

(i) $|Stab_{G}(\omega )N|=|Stab_{G}(\omega )||N|/|Stab_{G}(\omega ) \cap N|=|Stab_{G}(\omega )||N|/|Stab_{N}(\omega )|$\\

So $|Stab_{G}(\omega )N|/|Stab_{G}(\omega )|=|N|/|Stab_{N}(\omega )|$\\

Therefore,  $[Stab_{G}(\omega )N{:}Stab_{G}(\omega )]=[N{:}Stab_{N}(\omega )]$\\

So $|\omega^{G}|=[G{:}Stab_{G}(\omega )] = [G{:}Stab_{G}(\omega )N][Stab_{G}(\omega )N{:}Stab_{G}(\omega )]=$\

$[G{:}Stab_{G}(\omega )N][N{:}Stab_{N}(\omega )]=[G{:}Stab_{G}(\omega )N]|\omega^{N}|$. \\

Therefore, $|\omega^{G}|/|\omega^{N}|=[G{:}Stab_{G}(\omega )N]$. This proves (i).\\

(ii) $|\omega^{G}|/|\omega^{N}|=1$ iff $[G{:}Stab_{G}(\omega )N]=1$ iff $G=Stab_{G}(\omega )N$. This proves (ii).\\

(iii) $|\omega^{G}|/|\omega^{N}|=[G{:}N]$ iff $[G{:}Stab_{G}(\omega )N]=[G{:}N]$ iff $Stab_{G}(\omega )N=N$ iff $Stab_{G}(\omega) \leq N$. This proves (iii).\\

(iv) $|\omega^{G}|/|\omega^{N}|< [G{:}N]$ iff $[G{:}Stab_{G}(\omega )N] <  [G{:}N]$ iff  $|N| < |Stab_{G}(\omega )N|$ iff $Stab_{G}(\omega)$ is not a subgroup of $N$. This proves (iv).
\end{proof}

A particular instance of this Extended Frattini Argument is the following:

\begin{lemma}
Let $G$ be a finite group, let $N$ be a normal subgroup of $G$, and let $n \in N$. Then:\\

(i) $|n^{G}|/|n^{N}|=[G{:}C_{G}(n)N]$

(ii)  $|n^{G}|/|n^{N}|=1$ iff $G=C_{G}(n)N$

(iii)  $|n^{G}|/|n^{N}|=[G{:}N]$ iff $C_{G}(n) \leq N$

(iv) $|n^{G}|/|n^{N}|< [G{:}N]$ iff $C_{G}(n)$ is not a subgroup of $N$

\end{lemma}

\begin{proof}
Let $\omega = n$, let $\Omega = N$, and apply Lemma 1.3
\end{proof}

The group $Alt(8)$ is a simple subgroup of index $2$ in $Sym(8)$, and so we have $|Alt(8)|=\frac{1}{2} \cdot 8!=20{,}160$. $Alt(8)$ has two classes of involutions: the first with cycle structure $(ab)(cd)$ and cardinality $210$ and the second with cycle structure $(ab)(cd)(ef)(gh)$ and cardinality $105$. $Alt(8)$ also has two classes of elements of order $3$; cycle structures $(abc)$ and $(abc)(def)$ with cardinalities $112$ and $1{,}120$ respectively. $Alt(8)$ has two classes of elements of order $4$ with cycle structures $(abcd)(efgh)$ and $(abcd)(ef)$ and cardinalities $1{,}260$ and $2{,}520$ respectively. And finally, $Alt(8)$ has two classes of elements of order $6$ with cycle structures $(abcdef)(gh)$ and $(abc)(de)(fg)$ and cardinalities $3{,}360$ and $1{,}680$ respectively.\\

Note that $Alt(8)$ has only one class of cardinality $1{,}680$, since an element $a$ in such a class would satisfy $|C_{Alt(8)}(a)|=12$, and therefore the order of $a$ would be a divisor of $12$. The order of $a$ cannot be $1$, and it cannot be $12$ since $Alt(8)$ has no element of order $12$.  Therefore the order of $a$ is $2, 3, 4,$ or $6$. But among these orders, from the previous paragraph, only one has a class of cardinality $1{,}680$, the class of an element with cycle structure $(abc)(de)(fg)$ and order $6$. Similarly, $Alt(8)$ has only one class of cardinality $2{,}520$, since an element $a$ in such a class would satisfy  $|C_{Alt(8)}(a)|=8$, and therefore the order of $a$ would be a divisor of $8$. The order of $a$ cannot be $1$, and it cannot be $8$ since $Alt(8)$ has no element of order $8$.  So the order of $a$ is $2$ or $4$. But among these orders, only one has a class of cardinality $2{,}520$, the class of an element with structure $(abcd)(ef)$ and order $4$.\\

\begin{lemma} 
$Alt(8)$ is not a rational group.
\end{lemma}

\begin{proof} 
Let $a=(1234567)$, an element of order $7$ in $Alt(8)$. By the previous characterization of rational groups, it suffices to show that\\

 $[N_{Alt(8)}(\langle a \rangle){:} C_{Alt(8)}(\langle a \rangle)] \neq \phi(ord(a))$.\\

\noindent We have $C_{Alt(8)}(\langle a \rangle ) = \langle a \rangle $, and so $|C_{Alt(8)}(\langle a \rangle )|=7$. Therefore,\\

$[Alt(8){:} C_{Alt(8)}(\langle a \rangle )]= \frac{1}{7} (\frac{1}{2} 8!) \equiv 3 \hspace{0.1cm} (mod \hspace{0.1cm}  7)$.\\

But $\langle a \rangle $ is a Sylow $7$-subgroup of $Alt(8)$ and so by a Sylow theorem we have\\

$[Alt(8){:} N_{Alt(8)}(\langle a \rangle )] \equiv 1 (mod \hspace{0.1cm}  7)$. Therefore,\\

$[N_{Alt(8)}(\langle a \rangle ){:} C_{Alt(8)}(\langle a \rangle)] \equiv 3 (mod \hspace{0.1cm}  7)$, and so\\

 $[N_{Alt(8)}(\langle a \rangle){:}C_{Alt(8)}(\langle a \rangle)] \neq 6 =  \phi(7)= \phi(ord(a))$.
\end{proof}

Note that since $[N_{Alt(8)}(\langle a \rangle ){:}C_{Alt(8)}(\langle a \rangle)]$ is a divisor of $\phi(ord(a))=6$ and congruent to $3 \hspace{0.1cm} (mod \hspace{0.1cm}  7)$, it must be $3$. Therefore $Alt(8)$ does not induce an automorphism of order $2$ by conjugation on $\langle a \rangle $. Therefore, $a$ is not conjugate to $a^{-1}$ in $Alt(8)$.\\

\noindent Since $Alt(8)$ is not a rational group, it follows that $Alt(8)$ cannot satisfy hypothesis $\mathcal{H}$. Also, $Sym(8)$ does not satisfy hypothesis $\mathcal{H}$; in fact the permutations $(123456)$ and $(123456)(78)$ have the same centralizer in $Sym(8)$ (and therefore their conjugacy classes have the same cardinality) but since they have different cycle structures, they are not conjugate in $Sym(8)$.  Finally, we note that $Aut(Alt(8)) \cong Sym(8)$, and therefore $|Out(Alt(8))| =2$. \\

\noindent We can now prove the main theorem:

\begin{theorem} Let $G$ be a minimal counterexample to the $Sym(3)$ conjecture. Then $Soc(G)^{\prime}$ is not isomorphic to $Alt(8)$. 
\end{theorem}

\begin{proof} Let G be a minimal counterexample to the $Sym(3)$ conjecture; so $G$ satisfies the hypothesis $\mathcal{H}$, but $G$ is not isomorphic to $Sym(3)$ and $G$ is a group of least cardinality satisfying hypothesis $\mathcal{H}$ but not isomorphic to $Sym(3)$.  We wish to prove that $Soc(G)^{\prime }$ is not isomorphic to $Alt(8)$. Assume to the contrary that $Soc(G)'$ is isomorphic to $Alt(8)$. We wish to derive a contradiction.\\

\noindent Let $A=Soc(G)^{\prime }$ and let $B=C_{G}(A)$. $A=Soc(G)^{\prime}$ char $Soc(G) \trianglelefteq G$, and so $A  \trianglelefteq G$. Also, $B = C_{G}(A) \trianglelefteq N_{G}(A) = G$, and so $B \trianglelefteq G$. Therefore $A \cap B \trianglelefteq A$, so since $A$ is simple,  $A \cap B = 1$ or  $A \cap B =A$. But $A \cap B =A$ would imply $A \leq B = C_{G}(A)$ which would imply $A$ was abelian. So $A \cap B = 1$. Therefore $B$ is a proper subgroup of $G$ and $AB \cong A \times B$.\\

Claim: $B$ is finite, and $B \neq 1$.\\

\noindent Since $G$ is finite and $B$ is a subgroup of $G$, $B$ is finite, too. The homomorphism $\theta : G \rightarrow Aut(A)$ induced by conjugation has kernel $B$ and maps $AB$ onto $Inn(A)$, and so $G/B \cong Im(\theta )$ is a subgroup of $Aut(A)$ which contains $Inn(A)$. Since $Aut(A) \cong Sym(8)$ and $Inn(A) \cong Alt(8)$, it follows that either $G/B \cong Alt(8)$ or $G/B \cong Sym(8)$. Since neither $Alt(8)$ nor $Sym(8)$ satisfy hypothesis $\mathcal{H}$, it follows that we cannot have $B=1$. \\

Claim: $[G:AB]=2$.\\

\noindent Since $\theta $ maps $AB$ onto $Inn(A)$, then by the correspondence theorem we have \linebreak
$[G{:}AB]=[Im(\theta){:}Inn(A)]$, which divides $[Aut(A){:}Inn(A)]=[Sym(8){:}Alt(8)]=2$, and so $[G{:}AB]$ divides $2$. If $[G{:}AB]$ were to equal 1, we would have $G \cong A \times B$; but since $G$ is a rational group, this would imply that $A \times  B$, and therefore both $A$ and $B$, were rational groups. But $A$ is not a rational group, and so we must have $[G{:}AB]=2$.\\

Claim: $C_{G}(B) \leq AB$.\\

\noindent Suppose not. We wish to derive a contradiction. Then since $[G{:}AB]=2$, we would have $G=C_{G}(B)AB$. But since $B=C_{G}(A)$, we have $A \leq C_{G}(B)$, and so $G=C_{G}(B)AB = C_{G}(B)B$.  \\

\noindent We will consider two cases:  $Z(B)=1$ and $Z(B) \neq 1$.\\

\noindent Case 1: Assume that $Z(B)=1$. We have just shown that  $G=C_{G}(B)B$. Also, $C_{G}(B) \cap B=Z(B)=1$. Note that $C_{G}(B) \trianglelefteq N_{G}(B)=G$ and so both $C_{G}(B)$ and $B$ are normal in $G$. Thus $G \cong C_{G}(B) \times B$. In the proof of our first claim above, we have shown that either $G/B \cong Alt(8)$ or $G/B \cong Sym(8)$. But since $|G/B|=|G|/|B|=2|A||B|/|B|=2|A|=2|Alt(8)|$, we must have $G/B \cong Sym(8)$. So $C_{G}(B) \cong Sym(8)$ and so $G \cong Sym(8) \times B$. $Sym(8)$ does not satisfy hypothesis $\mathcal{H}$; in fact the permutations $s_{1}=(123456)$ and $s_{2}=(123456)(78)$ have the same centralizer in $Sym(8)$ but since they have different cycle structures they are not conjugate in $Sym(8)$, and therefore, since $G \cong Sym(8) \times B$, they are not conjugate in $G$. Therefore,\\

$C_{G}(s_{1})=C_{Sym(8)}(s_{1}) \times B = C_{Sym(8)}(s_{2}) \times B = C_{G}(s_{1})$\\

\noindent so\\

$|s_{1}^{G}|=[G{:}C_{G}(s_{1})]=[G{:}C_{G}(s_{2})]= |s_{2}^{G}|$.\\

\noindent But since $G$  satisfies hypothesis $\mathcal{H}$, and since the conjugacy classes $s_{1}^{G}$ and $s_{2}^{G}$ have the same cardinality, they must be the same conjugacy class, and so  $s_{1}$ and  $s_{2}$ must be conjugate in $G$, a contradiction.\\

\noindent Case 2: Assume that $Z(B) \neq 1$. Let $z \in Z(B)$ with $z \neq 1$. So $C_{G}(B) \leq C_{G}(z)$ and $B \leq C_{G}(z)$, so $G = C_{G}(B)B \leq C_{G}(z)$, so $z \in Z(G)$. But this would give us two distinct conjugacy classes $z^{G}=\{z\}$ and $1^{G}=\{1\}$ with the same cardinality, 1. But $G$ satisfies hypothesis $\mathcal{H}$, so this is a contradiction.\\

\noindent This proves the claim that $C_{G}(B) \leq AB$.\\

Claim: If $b_{1}$, $b_{2} \in B$ and if $|b_{1}^{B}|=|b_{2}^{B}|$, and if neither $C_{G}(b_{1}) \leq AB$ nor $C_{G}(b_{2}) \leq AB$, then $b_{1}$ and $b_{2}$ are conjugate in $B$.\\

\noindent Assume that $b_{1}$, $b_{2} \in B$ and that $|b_{1}^{B}|=|b_{2}^{B}|$, and that neither $C_{G}(b_{1}) \leq AB$ nor $C_{G}(b_{2}) \leq AB$. We wish to show that $b_{1}$ and $b_{2}$ are conjugate in $B$. Since \break  $[G{:}AB] =2$, $AB \trianglelefteq G$ and we have $C_{G}(b_{1})(AB) = G = C_{G}(b_{2})(AB)$. So by the Extended Frattini Argument, (using $AB$ for $N$), we get\\

$|b_{1}^{G}|/|b_{1}^{AB}| = 1 = |b_{2}^{G}|/|b_{2}^{AB}|$\\

\noindent and therefore $b_{1}^{AB}=b_{1}^{G}$ and $b_{2}^{AB}=b_{2}^{G}$. But\\

$|b_{1}^{(AB)}|=[AB{:}C_{(AB)}(b_{1})]=[AB{:}AC_{B}(b_{1})]$ and\\

$|b_{2}^{(AB)}|=[AB{:}C_{(AB)}(b_{2})]=[AB{:}AC_{B}(b_{2})]$\\

\noindent and so since $|C_{B}(b_{1})|=|C_{B}(b_{2})|$, it follows that $|b_{1}^{(AB)}| = |b_{2}^{(AB)}|$. So $|b_{1}^{G}|=|b_{2}^{G}|$. Since $G$ satisfies hypothesis $\mathcal{H}$, we get
$b_{1}^{G}=b_{2}^{G}$.  So\\

$b_{1}^{B}=b_{1}^{(AB)} = b_{1}^{G}=b_{2}^{G}=b_{2}^{(AB)} =b_{2}^{B}$\\

\noindent and so $b_{1}^{B} =b_{2}^{B}$, so $b_{1}$ and $b_{2}$ are conjugate in $B$. This proves the claim.\\

Claim: $B$ satisfies hypothesis $\mathcal{H}$.\\

\noindent We have already shown that $B$ is finite and non-trivial. So we now only need to show that distinct conjugacy classes of $B$ have distinct cardinalities. Assume to the contrary that there are elements $b_{1}$ and $b_{2}$ in $B$ such that $|b_{1}^{B}|=|b_{2}^{B}|$ but $b_{1}$ and $b_{2}$ are not conjugate in $B$. We wish to derive a contradiction.\\

\noindent Note that $|b_{1}^{B}|=|b_{2}^{B}|$ implies that $[B{:}C_{B}(b_{1})]=[B{:}C_{B}(b_{2})]$ and therefore implies that $|C_{B}(b_{1})|=|C_{B}(b_{2})|$.\\

\noindent Since $b_{1}$ and $b_{2}$ are not conjugate in $B$, then by the previous claim, we must have either $C_{G}(b_{1}) \leq AB$ or $C_{G}(b_{2}) \leq AB$. Without loss of generality, assume that $C_{G}(b_{1}) \leq AB$. Let $a_{1}$ be a 7-cycle of $Alt(8)$ in $A$ and let $a_{2}=a_{1}^{-1}$. Then $a_{1}$ and $a_{2}$ are not conjugate in $Alt(8)$, but they generate the same cyclic subgroup of $G$ of order $7$; since $G$ is a rational group, $a_{1}$ and $a_{2}$ are therefore conjugate in $G$. Therefore, $a_{1}^{A}$ is a proper subset of $a_{1}^{G}$. Also, $|C_{A}(a_{1})|=|C_{A}(a_{2})|$.   Since $B=C_{G}(A)$, $a_{1}^{A}=a_{1}^{AB}$. So by the Extended Frattini Argument,  $1 < |a_{1}^{G}|/|a_{1}^{AB}| = [G{:}C_{G}(a_{1})(AB)] \leq 2$.  So $[G{:}C_{G}(a_{1})(AB)] = 2$, so $G_{G}(a_{1}) \leq AB$. Similarly, $G_{G}(a_{2}) \leq AB$. Therefore (using $A \trianglelefteq G$, $B \trianglelefteq G$, and $AB \cong A \times B$), we get:\\

\noindent $C_{G}(a_{1}b_{1})=C_{G}(a_{1}) \cap C_{G}(b_{1})= C_{AB}(a_{1}) \cap C_{AB}(b_{1}) = C_{A}(a_{1}) \times C_{B}(b_{1})$.\\

\noindent Similarly,\\

\noindent $C_{G}(a_{1}b_{2})=C_{G}(a_{1}) \cap C_{G}(b_{2})= C_{AB}(a_{1}) \cap C_{AB}(b_{2}) = C_{A}(a_{1}) \times C_{B}(b_{2})$.\\

\noindent Thus $|C_{G}(a_{1}b_{1})|=|C_{A}(a_{1})| \times |C_{B}(b_{1})| = |C_{A}(a_{1})| \times |C_{B}(b_{2})|= |C_{G}(a_{1}b_{2})|$.\\

\noindent So $|(a_{1}b_{1})^{G}|=|(a_{1}b_{2})^{G}|$. Since $G$ satisfies hypothesis $\mathcal{H}$, it follows that $a_{1}b_{1}$ is conjugate to $a_{1}b_{2}$ in $G$. Choose $g \in G$ such that $(a_{1}b_{1})^{g}=a_{1}b_{2}$. Then $a_{1}^{g}b_{1}^{g}=a_{1}b_{2}$. Therefore (again using $A \trianglelefteq G$, $B \trianglelefteq G$, and $AB \cong A \times B$) we get $a_{1}^{g}=a_{1}$ and $b_{1}^{g}=b_{2}$. This implies that $g \in C_{G}(a_{1})$ and so $g \in AB$. But no element $ab$ of $AB$ can conjugate $b_{1}$ to $b_{2}$, because since $A$ centralizes $B$, this would imply that $b$ conjugates $b_{1}$ to $b_{2}$, a contradiction to our assumption that $b_{1}$ and $b_{2}$ are not conjugate in $B$. This proves the claim that $B$ satisfies hypothesis $\mathcal{H}$.\\

Claim: $B \cong Sym(3)$.\\

\noindent $B$ is a finite non-trivial proper subgroup of $G$ which satisfies the hypothesis $\mathcal{H}$ of the $Sym(3)$ conjecture. Since we are assuming that $G$ is a counterexample of least cardinality to the $Sym(3)$ conjecture, this implies that $B \cong Sym(3)$. This proves the claim.\\

\noindent Therefore $G$ has the normal subgroup $AB \cong Alt(8) \times Sym(3)$ of index $2$. Since $B \cong Sym(3)$, let $\sigma $ be an involution of $B$ and let $\delta $ be an element of order $3$ of $B$.\\

\noindent $Alt(8)$ has a unique conjugacy class of cardinality $1{,}680$ which is the class of an element $s$ of $Alt(8)$ of order $6$. Also, $Alt(8)$ has a unique conjugacy class of cardinality $2{,}520$ which is the class of an element $f$ of $Alt(8)$ of order $4$.\\

\noindent Since $Sym(3) \cong B \trianglelefteq G$, $\sigma ^{G} \subseteq B$ and so $\sigma ^{G}=\sigma ^{B}$ and so by the Extended Frattini Argument, we have $G=C_{G}(\sigma )B$. So $C_{G}(\sigma )$ cannot be contained in $AB$. Similarly, we know that $C_{G}(\delta )$ cannot be contained in $AB$. Also, we have $s^{A} \subseteq s^{G}$ and $s^{G}$ is a union of $A$-conjugacy classes of elements of $A$ of order $6$ and thus has cardinality $1{,}680$ or $5{,}040$. Also, $B \leq C_{G}(s)$ and so $AB \leq C_{G}(s)A$. Then\\

$|s^{G}|/|s^{A}| = [G{:}C_{G}(s)A] \leq [G{:}AB]=2$, so\\

\noindent $|s^{G}| \leq 2|s^{A}|=3{,}360$. So $|s^{G}|=1{,}680=|s^{A}|$. So $s^{G}=s^{A}$.  Similarly, we have $f^{A} \subseteq f^{G}$ and $f^{G}$ is a union of $A$-conjugacy classes of elements of $A$ of order $4$ and thus has cardinality $1{,}260$ or $3{,}780$. Also, $B \leq C_{G}(f)$ and so $AB \leq C_{G}(f)A$. Then\\

$|f^{G}|/|f^{A}| = [G{:}C_{G}(f)A] \leq [G{:}AB]=2$,  so\\

\noindent $|f^{G}| \leq 2|f^{A}|=2{,}520$. So $|f^{G}|=1{,}260=|f^{A}|$. So $f^{G}=f^{A}$.  Since $s^{G}=s^{A}$ and  $f^{G}=f^{A}$, we know from the Extended Frattini Argument that neither $C_{G}(s)$ nor $C_{G}(f)$ can be contained in $AB$.\\

Claim: $C_{G}(s \sigma)$ is not contained in $AB$.\\

\noindent Since $C_{G}(s)$ is not contained in $AB$, we can choose $g \in C_{G}(s)$ such that $g \notin A \times B$. Since  $Sym(3) \cong B \trianglelefteq G$, $g$ permutes the three involutions of $B$. But $g$ cannot fix all three for otherwise we would have $g \in C_{G}(B)$ and by the previous claim, $C_{G}(B) \leq AB$. So either $g$ fixes one involution and transposes the other two, or $g$ permutes the three involutions in a $3$-cycle. But $g$ cannot permute the involutions in a $3$-cycle, because in this case, $g^{3}$ (which is also not in $AB$, since $[G{:}AB]=2$) would be in $C_{G}(B)$, which by a previous claim is a subgroup of $AB$.  So $g$ must fix one involution and transpose the other two. Without loss of generality, we assume that $\sigma $ is the involution of $B$ fixed by $g$. Then $g \in C_{G}(s) \cap G_{G}(\sigma )$. Thus, $g \in C_{G}(s \sigma )$. So $g$ is an element of $C_{G}(s \sigma)$ which is not an element of $AB$. This proves the claim that $C_{G}(s \sigma)$ is not contained in $AB$.\\

Claim: $C_{G}(f \delta )$ is not contained in $AB$.\\

\noindent Since $C_{G}(f)$ is not contained in $AB$, we can choose $h \in C_{G}(f)$ such that $h \notin AB$. Since  $Sym(3) \cong B \trianglelefteq G$, $h$ permutes the two elements of order $3$ of $B$. If $h$ fixes both of them, then $h \in C_{G}( \delta )$, and so we can take $h$ as our element in $C_{G}(f \delta )$ which is not in $AB$. Otherwise, assume that $h$ transposes the two elements of order $3$ of $B$. Since $C_{G}(\delta )$ cannot be contained in $AB$, there is an element $h^{\prime } \in G$ which is not in $AB$ which centralizes $\delta $. Every element in $G$ which is not $AB$ can be written as $hab$, so write $h^{\prime }=hab$. Then $(\delta )^{hab}= \delta $ so since $ (\delta )^h \in B$, and $A$ centralizes $B$, we get $(\delta )^{ha}= (\delta )^{h}$ so $(\delta )^{hab} = (\delta )^{hb}$. So $hb$ centralizes $\delta $. But $hb$ also centralizes $f$ (since both $h$ and $b$ do), and so $hb$ is in in $C_{G}(f \delta )$. And $hb$ is not in $AB$ (since $b$ is, but $h$ isn't). So $hb$ is an element of $C_{G}(f \delta )$ which is not an element of $AB$. This proves the claim that $C_{G}(f \delta )$ is not contained in $AB$.\\

\noindent Thus, $|C_{G}(s \sigma )| = 2|C_{A}(s) \times C_{B}(\sigma )| = 2|C_{A}(s)| \times |C_{B}(\sigma )|=2(12)(2)=48$.\\

\noindent Also, $|C_{G}(f \delta )| = 2|C_{A}(f) \times C_{B}(\delta )| = 2|C_{A}(f)| \times |C_{B}(\delta )| = 2(8)(3) = 48$.\\

\noindent Thus,  $|C_{G}(s \sigma )| = 48 = |C_{G}(f \delta )|$. Therefore the conjugacy classes in $G$ of $s \sigma $ and $f \delta $ have the same cardinality. Since $G$ satisfies hypothesis $\mathcal{H}$, it follows that $s \sigma $ and $f \delta $ are conjugate in $G$. But $s \sigma $ has order $6$ while $f \delta $ has order $12$, and so they cannot be conjugate in $G$. So we have reached the desired contradiction. It follows that $Soc(G)^{\prime}$ is not isomorphic to $Alt(8)$. 
\end{proof}

\bibliographystyle{amsplain}

\end{document}